\documentclass[11pt]{article}

\usepackage{tikz}
\usepackage{subfigure}
\usepackage[english]{babel}

\usepackage[center]{caption2}
\usepackage{amsfonts,amssymb,amsmath,latexsym,amsthm}
\usepackage{multirow}
\usepackage[usenames,dvipsnames]{pstricks}
\usepackage{epsfig}
\usepackage{pst-grad} 
\usepackage{pst-plot} 
\usepackage[space]{grffile} 
\usepackage{etoolbox} 
\makeatletter 
\patchcmd\Gread@eps{\@inputcheck#1 }{\@inputcheck"#1"\relax}{}{}

\def\qedB{{\hfill\enspace\vrule height8pt depth0pt width8pt}}

\newtheorem{thm}{Theorem}

\newtheorem{lem}[thm]{Lemma}

\newtheorem{conj}[thm]{Conjecture}

\newtheorem{claim}{Claim}

\begin{document}
\title{\bf\Large A conjecture of Verstra\"ete on vertex-disjoint cycles}
\date{}

\author{
Jun Gao\footnote{Email: gj0211@mail.ustc.edu.cn.}~~~~~~~
Jie Ma\footnote{Email: jiema@ustc.edu.cn.
Research supported in part by NSFC grants 11501539 and 11622110.}\\
\medskip \\
School of Mathematical Sciences\\
University of Science and Technology of China\\
Hefei, Anhui 230026, China.
}

\maketitle

\begin{abstract}
Answering a question of H\"aggkvist and Scott, Verstra\"ete proved that every sufficiently large graph with average degree at least $k^2+19k+10$ contains $k$ vertex-disjoint cycles of consecutive even lengths.
He further conjectured that the same holds for every graph $G$ with average degree at least $k^2+3k+2$.
In this paper we prove this conjecture for $k\geq 19$ when $G$ is sufficiently large.
We also show that for any $\epsilon>0$ and large $k\geq k_\epsilon$, 
average degree at least $k^2+3k-2+\epsilon$ suffices,
which is asymptotically tight for infinitely many graphs.
\end{abstract}

\section{Introduction}
Throughout this paper, all graphs considered are simple and the word {\it disjoint} will always mean for {\it vertex-disjoint} unless otherwise specified.

A classic result of Corradi and Hajnal \cite{CH} says that for any $k\geq 2$, every graph of order at least $3k$ and minimum degree at least $2k$ contains $k$ disjoint cycles.
Thomassen \cite{T83} proved that for any $k\geq 2$, there exists some $n_k$ such that every graph of order at least $n_k$ and minimum degree at least $3k+1$ contains $k$ disjoint cycles of the same length.
He also conjectured in \cite{T83} that to assure the existence of $k$ disjoint cycles of the same length, it suffices for graphs of sufficiently large order and minimum degree at least $2k$
(the case $k=2$ was conjectured earlier by H\"aggkvist; see \cite{E96,T83}).
This was confirmed by Egawa \cite{E96} for $k\geq 3$ and later by Verstra\"ete \cite{V03} for $k\geq 2$.
In \cite{HS}, H\"aggkvist and Scott asked whether there exists a quadratic function $q(k)$ such that every graph with minimum degree at least $q(k)$ contains $k$ disjoint cycles of consecutive even lengths.
Verstra\"ete \cite{V02} answered this in the affirmative by proving that for any $k\geq 2$,
every graph of order at least $n_k = 16(k^2)!$ and average degree at least $k^2 +19k +10$ contains $k$ disjoint cycles of consecutive even lengths.
This is tight up to the $O(k)$ term.
He also made the following conjecture.

\begin{conj}[Verstra\"ete \cite{V02}]\label{conj}
Any graph of average degree at least $(k+2)(k+1)$ contains $k$ vertex-disjoint cycles of consecutive even lengths.
\end{conj}

In this paper, we prove this conjecture for $k\geq 19$ when the graph is sufficiently large.

\begin{thm}\label{THM: main}
Let $k$ be an integer at least 19 and let $G$ be a graph of order at least $n_k=2^{32k^3}$ and average degree at least $(k+2)(k+1)$.
Then $G$ contains $k$ disjoint cycles of consecutive even lengths.
\end{thm}

Let $s=\frac{1}{2}(k^2+3k)$.
We now observe that for all $n$, the complete bipartite graph $K_{s-1,n-s+1}$ does not contain $k$ disjoint cycles of consecutive even lengths,
while its average degree equals $2(s-1)(n-s+1)/n=k^2+3k-2-\epsilon_n$, where $\epsilon_n>0$ goes to zero as $n$ goes to infinity.
This shows that for any positive real number $d<k^2+3k-2$, average degree at least $d$ cannot force the existence of such $k$ disjoint cycles.
Being an asymptotic result, we prove that in contrast of the above example, average degree at least $k^2+3k-2+\epsilon$ will suffice.

\begin{thm}\label{THM:largek}
For every $\epsilon >0$, there exists $k_\epsilon$ such that the following holds for any $k\geq k_\epsilon$.
If $G$ is a graph of order at least $n_k$ and average degree at least $k^2+3k-2+\epsilon$,
then $G$ contains $k$ disjoint cycles of consecutive even lengths.
\end{thm}

We define some notations.
Let $G$ be a graph.
For $S\subseteq V(G)$, let $G[S]$ be the subgraph of $G$ induced on the vertex set $S$.
Let $A,B\subseteq V(G)$ be disjoint sets.
We denote $(A,B)$ to be the set of edges between $A$ and $B$ and $e(A,B) = |(A,B)|$.
Let $G(A,B)$ be the bipartite subgraph of $G$ spanned by $(A,B)$ and $G[A,B)$ be the subgraph of $G$ spanned by $(A,B)\cup E(G[A])$.
An $A$-$B$ path means a path with one endpoint in $A$ and other in $B$.
We also write $[t]:=\{1,2,...,t\}$ for any integer $t\geq 1$.

\medskip

The rest of the paper is organized as follows.
In Section 2, we collect and establish some lemmas.
We then prove Theorem \ref{THM: main} and Theorem \ref{THM:largek} in Sections 3 and 4, respectively
(for a sketch of the proofs, we direct readers to the beginning of Section 3).
In Section 5, we provide a weaker bound for the case $k=2$ and conclude the paper by a question.


\section{Preliminaries}
In this section we prepare some lemmas for the coming sections.
The first lemma is a user-friendly weaker form of the classic theorem of K\H{o}v\'ari-S\'os-Tur\'an \cite{KST} (also see Lemma 4 in \cite{V02}).

\begin{lem}\label{LEM:Kss}
Let $\delta>0$ be any real, $s$ be any natural number and $G$ be a bipartite graph with bipartition $(A,B)$.
If $e(G)\geq (s-1+\delta)|A|$ and $\delta |A|\geq |B|^{s}$, then $G$ contains a copy of $K_{s,s}$.
\end{lem}
\begin{proof}
Suppose that $G$ doesn't contain $K_{s,s}$. Let $N$ denote the number of stars $K_{1,s}$ in $G$ with centers in $A$.
By the standard double-counting argument, we have
\begin{align*}
|B|^s>(s-1)\binom{|B|}{s}\geq N=\sum_{v\in A} \binom{d_G(v)}{s}\geq \delta |A|\geq |B|^s,
\end{align*}
where the second last inequality holds,
because under the condition $e(G)\geq (s-1+\delta)|A|$ and by convexity,
$\sum_{v\in A} \binom{d_G(v)}{s}$ is minimized when $(1-\delta)|A|$ vertices in $A$ have degree $s-1$ and other vertices in $A$ have degree $s$.
This contradiction completes the proof.
\end{proof}

The coming useful lemma can be found implicitly in \cite{BS} and explicitly in \cite{V00}.

\begin{lem}[\cite{BS,V00}]\label{A-B path}
Let $H$ be a graph comprising a cycle with a chord. Let $(A,B)$ be a non-trivial partition of $V(H)$.
Then $H$ contains $A$-$B$ paths of every length less than $|H|$, unless $H$ is bipartite with bipartition $(A,B)$.
\end{lem}

To apply this, we often use the following lemma to get a long cycle with a chord.

\begin{lem}[\cite{V00}]\label{cycle with chord}
Let $k\geq 2$ be a natural number and $G$ be a graph of average degree at least $2k$ and girth $g$.
Then $G$ contains a cycle of length at least $(g-2)k+2$, with at least one chord.
\end{lem}

The next lemma will be used to find appropriate-size cycles (not necessarily disjoint) of consecutive even lengths in dense graphs.
This follows the approach of \cite{V02} in spirit and provides a key ingredient for the proofs of the coming sections.
Instead of adapting the route in \cite{V02} (i.e., the use of Theorem 1 and Lemma 3 of \cite{V02}),
we prove the following to improve the resulting constant coefficients (by a factor of two for general $k$).
In the case $k=2$, we undertake a more careful analysis, which we hope will shed some light on the resolution of Conjecture \ref{conj}
for small $k$ and perhaps some other related problems.

\begin{lem}\label{LEM:main}
Let $\epsilon$ be any positive real, $k\geq 2$ be a natural number and $G$ be an $n$-vertex graph.
Suppose that the average degree of $G$ is at least $8k+4\epsilon$ for $k\geq 3$ or at least $5k+2\epsilon$ for $k=2$.
Then there exist $k$ cycles of consecutive even lengths in $G$, the shortest one of which has length at most $2\log_{1+\epsilon/k}n+2$.
\end{lem}

{\noindent \it Proof.}
First let us consider for $k\geq 3$.
It is clear that $G$ contains a bipartite subgraph $H$ with average degree at least $4k+2\epsilon$.	
We choose such $H$ with the minimum $|V(H)|$.
Then it holds for any $S\subseteq V(H)$,
\begin{align}\label{neq:1}
e(H[S]) +e(V(H)\setminus S,S)>(2k+\epsilon)|S|,
\end{align}
as, otherwise $e(H[V(H)\setminus S])\geq (2k+\epsilon)|V(H)\setminus S|$, contradicting the minimality of $V(H)$.

Let $t=\log_{1+\epsilon/k}n+1$.
We may assume that $H$ doesn't contain cycles of $k$ consecutive even lengths, the shortest of which has length at most $2t$.	
Fix a vertex $r$ in $H$ and let $L_i$ denote the set of vertices at distance $i$ from $r$ in $H$.
For $i\geq 1$, let $H_i= H[\cup _{0\leq j\leq i} L_j]$.

We claim that $e(L_i,L_{i+1})\leq k(|L_i| +|L_{i+1}|)$ for any $i\leq t-1$.
Suppose for a contradiction that there exists some $\ell\leq t-1$ with $e(L_\ell,L_{\ell+1})> k(|L_i| +|L_{i+1}|)$.
By Lemma \ref{cycle with chord}, we can find $R\subseteq H[L_\ell \cup L_{\ell+1}]$ which comprises a cycle of length at least $2k+2$ plus a chord.
Let $T$ be the minimal subtree of a BFS tree with root $r$ in $H_\ell$ such that $T$ contains $V(R)\cap V_\ell$.
Let $A$ be the set of vertices of $R$ in one branch of $T$ and let $B=V(R)\setminus A$.
By the minimality of $T$, $(A,B)$ cannot be the bipartition of $R$.
By Lemma \ref{A-B path}, there are $A$-$B$ paths of all lengths up to $2k$.
It is then clear that all $A$-$B$ paths of even lengths say $2,4,...,2k$ have one endpoint in $A$ and the other in $L_{\ell}\cap (V(R)\setminus A)$.
This gives $k$ cycles $C_{2r+2}, C_{2r+4},...,C_{2r+2k}$ of consecutive even lengths in $H$, where $r$ is the distance form $L_\ell$ to the root of $T$ and thus $r\leq \ell\leq t-1$.
This proves the claim.
	
By this claim and by \eqref{neq:1} (using $S=V(H_i)$), for all $i\leq t-1$ we have
\begin{equation*}
\begin{split}
(2k+\epsilon) \sum_{j=0}^{i}|L_j| &\leq e(L_i,L_{i+1}) +e(H_i)= \sum_{j=0}^{i} e(L_j,L_{j+1})\\
&\leq \sum_{j=0}^{i}k(|L_j|+|L_{j+1}|)\leq k|L_{i+1}|+2k \sum_{j=0}^{i}|L_j|
\end{split}
\end{equation*}
This implies that for all $i\leq t-1$, $|L_{i+1}| \geq \frac{\epsilon}{k} |V(H_i)|$ and thus $|V(H_{i+1})|\geq (1+\frac{\epsilon}k)|V(H_i)|$.
So $|V(G)|\geq |V(H_t)|\geq (1+\frac{\epsilon}{k})^{t}>n$, a contradiction.
This finishes the proof for $k\geq 3$.

Now consider $k=2$. Let $G$ be an $n$-vertex graph with $e(G)\geq (5+\epsilon)n$.
Without loss of generality, we may assume that $G$ has at least $(5+\epsilon)|V(G)|$ edges and subject to this, $|V(G)|$ is the minimum.
Similarly as \eqref{neq:1}, we can derive that for any $S\subset V(G)$,
\begin{align}\label{neq:2}
e(G[S]) +e(V(G)\setminus S,S)>(5+\epsilon)|S|.
\end{align}
Let $t=\log_{1+\epsilon/2}n+1$.
Fix a vertex $r$ in $G$ and let $L_i$ denote the set of vertices at distance $i$ from $r$ in $G$.
Also for $i\geq 1$, let $G_i= G[\cup _{0\leq j\leq i} L_j]$. Note that $G[L_i]$ may contain edges.

First we claim that for any $i\leq t-1$, we have
\begin{align}\label{neq:3}
e(L_i,L_{i+1})\leq |L_i|+2|L_{i+1}|.
\end{align}
Suppose for a contradiction that $e(L_i,L_{i+1})\geq |L_i|+2|L_{i+1}|+1$ for some $i<t$.
By standard deletion arguments, there exists a nonempty connected bipartite subgraph $H\subseteq G(L_i,L_{i+1})$ such that
any vertex in $V(H)\cap L_i$ has degree at least 2 in $H$ and any vertex in $V(H)\cap L_{i+1}$ has degree at least 3 in $H$.
Let $T$ be the minimal subtree of a BFS tree with root $r$ in $G_i$ such that $T$ contains $V(H)\cap L_i$.
Let $A$ be the set of vertices of $V(H)\cap L_i$ in one branch of $T$ and let $B=(V(H)\cap L_i)\setminus A$.
Since $H$ is connected, there exists some vertex $y\in L_{i+1}$ with neighbors in both $A$ and $B$.
As $d_H(y)\geq 3$, we may assume that $x_1,x_2\in N(y)\cap A$ and $x_3\in N(y)\cap B$.
Since $x_1$ has degree at least two in $G$, let $y'\in N_H(x_1)\setminus\{y\}$.
Consider $x_3'\in N_H(y')\setminus \{x_1,x_2\}$.
In either case that $x_3'\in A$ or $x_3'\in B$,
we can find a path $aca'c'b$ on five vertices with $a,a'\in A$, $b\in B$ and $c,c'\in L_{i+1}$.
By the choice of $T$, this gives two cycles of consecutive even lengths, the shortest of which has length at most $2t$, proving \eqref{neq:3}.

Next we claim that for any $i\leq t-1$, we have
\begin{align}\label{neq:4}
e(G[L_i]) \leq 2|L_i|.
\end{align}
Suppose that $e(L_i)\geq 2|L_i|+1$ for some $i\leq t-1$.
We may further assume that the minimum degree in $G[L_i]$ is at least 3.
Let $R$ be a component of $G[L_i]$ with $e(R)\geq 2|V(R)|+1$.
Let $T$ be the minimal subtree of a BSF tree with root $r$ in $G_i$ such that $T$ contains $V(R)$.
Let $A$ be a set of vertices of $R$ in one branch of $T$ and let $B=V(R)\setminus A$.

If there exists a path $a_1a_2a_3a_4b$ in $R$ with $a_i\in A$ for $i\in [4]$ and $b\in B$,
then by using the subtree $T$, one can find two desired cycles of consecutive even lengths in $G$.
So we may assume that there is no such path in $R$, from which one can also conclude that
neither $R[A]$ or $R[B]$ can contain any cycle of length at least four or any path on six vertices.
Suppose that $e(R[A])\geq |A|+1$.
By the above propositions, it follows that $R[A]$ must contain a subgraph $R'$ consisting of two triangles with a common vertex.
As the minimum degree in $R$ is at least 3, by considering the vertices in $R'$ of degree two,
one would derive one of the subgraphs forbidden in above, a contradiction.
So $e(R[A])\leq |A|$ and similarly $e(R[B])\leq |B|$.
This shows that $e(R(A,B))\geq |A|+|B|+1$.

If there is a path $a_1a_2b_1a_3b_2$ in $R$ such that $a_1,a_2,a_3 \in A$ and $ b_1,b_2 \in B$,
then it is easy to see that $G$ contains two desired consecutive even cycles.
So $R$ doesn't contain such a path (call it a {\it forbidden} path).
We see that $R(A,B)$ contains an even cycle $C$.
We assert that $|C|\geq 6$.
Suppose that $C$ is a four-cycle, say $a_1b_1a_2b_2a_1$ with $a_1,a_2\in A$ and $b_1,b_2\in B$.
Let $x\in N_R(a_1)\setminus \{b_1,b_2\}$.
If $x\in B\setminus \{b_1,b_2\}$, then the path $P$ in $T$ between $x$ and $b_1$ gives
two desired cycles $P\cup xa_1b_1$ and $P\cup xa_1b_2a_2b_1$ of consecutive even lengths.
If $x\in A\setminus \{a_2\}$, then we get a forbidden path.
This in fact shows that any vertex in $C$ cannot have neighbors outside of $C$,
implying that $R=V(C)$ and thus contradicting that $e(R)\geq 2|V(R)|+1$.

If $A$ or $B$ contains an edge, since $|C|\geq 6$ and $R$ is connected,
it is easy to see that there always exists a forbidden path in $R$.
So we may assume that $e(A)=e(B)=0$.
Take the minimal subtree $T'$ of $T$ containing $A$ and view $B$ as the next level of $T'$.
Running the same proof for \eqref{neq:3}, one would get $e(R(A,B))\leq |A|+2|B|$.
But $e(R(A,B))=e(R)\geq 2|A|+2|B|+1$.
This final contradiction proves \eqref{neq:4}.

Now combining \eqref{neq:2}, \eqref{neq:3} and \eqref{neq:4}, for any $i\leq t-1$ we have
\begin{equation*}
\begin{split}
(5+\epsilon) \sum_{j=0}^{i}|L_j| &\leq e(L_i,L_{i+1}) +e(\cup_{j\leq i} L_j)\leq \sum_{j=0}^{i} (e(L_j,L_{j+1})+e(L_j))\\
&\leq \sum_{j=0}^{i}(3|L_j|+2|L_{j+1}|)\leq 2|L_{i+1}|+ 5\sum_{j=0}^{i} |L_j|
\end{split}
\end{equation*}
Then for any $i\leq t-1$, we have $|L_{i+1}|\geq \frac{\epsilon}{2} |V(G_i)|$ and thus $|V(G_{i+1})|\geq (1+\frac{\epsilon}{2})|V(G_i)|$.
This implies a contradiction that $|V(G)|\geq |V(G_t)|\geq (1+\frac{\epsilon}{2})^{t}> n$, proving the lemma.
\qedB

\section{Proof of Theorem~\ref{THM: main}}
Let $k\geq 19$ and $G$ be a graph of order $n\geq n_k=2^{32k^3}$ and average degree at least $k^2 +3k +2$.
Assume that $G$ doesn't contain $k$ disjoint cycles of consecutive even lengths.

\medskip
{\noindent \bf Outline of the proof.} We begin with a sketch of the proof. Set $H:=G$.
Following the approach in \cite {V02}, we will repeatedly apply Lemma \ref{LEM:main}
on $H$ to get $k$ consecutive even cycles (say $C_1,...,C_k$) of bounded lengths and then update $H:=H-\cup_{i=1}^k V(C_i)$,
until the average degree of $H$ is small.
This will yield a partition $V(G)=V_1\cup V_2$ satisfying that $|V_2|=o(n)$ and a considerable amount of edges of $G$ lie in $(V_1,V_2)$.
If $e(V_1,V_2)$ is large enough, then by Lemma \ref{LEM:Kss} we find $K_{t,t}\subseteq (V_1,V_2)$ for some large $t$, which would complete the proof.
So $e(V_1)$ must be $\Omega(n)$.
A new and crucial observation here is that vertices of $V_2$ with a large number of neighbors in $V_1$
would help building the desired even cycles a lot.
On the other hand, if such vertices in $V_2$ are few,
then one will also benefit as the relevant density between $V_1$ and $V_2$
will increase in the recursive process of deleting certain disjoint cycles which have been found.
This paradox will be demonstrated with details in two separated cases, depending on if $e(V_1)$ is relatively big or just of intermediate size.

\medskip

To be precise, let $t=\log_{1+1/4k}n+1$ and we define a sequence of subgraphs $G_0\supseteq G_1\supseteq ... \supseteq G_m$ as following.
Let $G_0:= G$. Suppose that we have defined $G_i$ for some $i\geq 0$.
Denote $r_i$ to be the minimum integer $r$ such that $G_i$ contains $k$ cycles of lengths $2r,2r +2,...,2r+2k-2$
(in case that there is no $k$ cycles of consecutive even lengths, let $r_i=\infty$).
Let $X_i$ be a union of vertex-sets of $k$ cycles of lengths $2r_i,2r_i+2,...,2r_i+2k-2$ in $G_i$.
If $r_i\leq t$, then let $G_{i+1}=G_i-V(X_i)$; otherwise, we terminate (say at $G_m$).

Write $V_1'= V(G_m)$ and $V_2'= V(G)\setminus V_1'$.
Note that $2\leq r_i\leq t$ for each $i\in \{0,1,...,m-1\}$.
Among all defined $r_i$'s, if $k$ of them are identical,
then clearly $G$ contains $k$ disjoint cycles of consecutive even lengths, a contradiction.
So we have $m\leq tk$ and $|V_2'|\leq 2(t+k)k\cdot m\leq 2k^2t(t+k).$
Let $$U=\left\{v\in V_1': d_G(v)\geq \frac{n}{\log_2n} \right\}, ~~~ V_1 = V_1' \setminus U \text{~~ and ~~} V_2 = V_2' \cup U.$$
We see that $G[V_1']=G_m$ doesn't contain $k$ cycles of consecutive even lengths,
where the shortest cycle has length at most $2t=2\log_{1+1/4k}n +2$.
By Lemma \ref{LEM:main}, $G[V_1']$ or any its subgraph (such as $G[V_1]$) has average degree at most $8k+1$.
Then it holds
$$|U|\cdot \frac{n}{\log_2n}\leq \sum_{v\in U} d_G(v)\leq 2e(G[V_1'])+|U||V_2'|\leq (8k+1)n+|U||V_2'|,$$
implying that $|U|\leq (16k+2)\log_2 n$. Therefore, we have
\begin{align}\label{equ:V2}
|V_2|=|U|+|V_2'|\leq k^4 (\log_{1+1/4k}n)^2 \text{~~ and ~~} e(G[V_2])\leq |V_2|^2/2\leq n/8.
\end{align}
We also collect the properties of $G[V_1]$ that
\begin{align}\label{equ:V1}
G[V_1] \text{ has average degree at most } 8k+1 \text{ and any } v\in V_1 \text{ has } d_G(v)\leq n/\log_2n.
\end{align}
Next we claim that
\begin{align}\label{equ:e1}
e(G[V_1])> 7n/8.
\end{align}
Otherwise we have $e(G[V_1]) \leq 7n/8$. By \eqref{equ:V2}, it then follows that
\begin{equation*}
e(V_1,V_2)\geq \frac{1}{2}(k^2+3k+2)n-\frac{7}{8}n-\frac{1}{8}n\geq \frac{1}{2}(k^2+3k)|V_1|
\end{equation*}
Let $s =\frac{1}{2}(k^2+3k)$. Since $|V_1|\geq |V_2|^{s}$, by Lemma~\ref{LEM:Kss} (with $\delta=1$),
$G$ contains a copy of $K_{s,s}$ and thus $G$ contains $k$ disjoint cycles of lengths $4,6,...2k+2$. This proves \eqref{equ:e1}.

The rest of the proof will be divided into two cases, depending on whether $e(G[V_1]) \leq (2k+1)n$ or not.
We distinguish in two subsections.

\subsection {$e(G[V_1])\leq (2k+1)n$}\label{subs:3.1}
Set $e(G[V_1]):= \left(\frac{5}{8} + \frac{\epsilon(8k+2)}{k}\right) n$.
By \eqref{equ:e1}, we have $\frac{1}{32k+8}<\frac{\epsilon}{k}<\frac{1}{4} $.
Suppose there are exactly $m$ vertices $v\in V_2$ such that $|N(v)\cap V_1|>(1- \frac{\epsilon}{k})n$.
	
\begin{claim}\label{Claim:I-m}
$\frac{5k}{2}<m<\frac{(k+3)k}{2}.$
\end{claim}
	
\begin{proof}[Proof of Claim \ref{Claim:I-m}]
We first show $m<(k+3)k/2$. Otherwise there exist vertices $x_i\in V_2$ for $1\leq i\leq (k+3)k/2$ with $|N(x_i)\cap V_1|> (1- \frac{\epsilon}{k})n>3n/4$.
Any two of these vertices have at least $n/2$ common neighbors in $V_1$,
so one can find $k$ disjoint cycles of lengths $4,6,...,2(k+2)$ in $(V_1,V_2)$ (i.e., greedily constructing these cycles one at a time using vertices $x_i$'s).

To prove $m>\frac{5k}{2}$, we will need to show that $e(V_1,V_2)\leq mn + (1- \frac{\epsilon}{k})(\frac{k^2+3k}{2}-m)n+\frac{n}4$.
Suppose for a contradiction that $e(V_1,V_2)> mn + (1- \frac{\epsilon}{k})(\frac{k^2+3k}{2}-m)n+\frac{n}4$.
Let $C_1,C_2,...,C_t$ be the maximal collection of $t$ disjoint cycles in $(V_1,V_2)$ with $|C_j|=2k+4-2j$.
Clearly we have $t\leq k-1$.
Let $R_i=V_i-V(C_1)\cup...\cup V(C_t)$ for $i\in \{1,2\}$.
Then $(R_1,R_2)$ doesn't contain any cycle of length $2k+2-2t$ and $|V_i\setminus R_i|=\frac{1}{2}(2k+3-t)t$ for $i\in \{1,2\}$.
We have
\begin{equation*}
\begin{split}
e(R_1,R_2)&\geq e(V_1,V_2)-\sum_{x\in V_2\setminus R_2}d_{V_1}(x)-\sum_{y\in V_1\setminus R_1} d_{V_2}(y)\\
&\geq e(V_1,V_2)- \left((1- \frac{\epsilon}{k})\frac{(2k+3-t)t}{2}+\frac{\epsilon}{k}m\right)n-\frac{(2k+3-t)t}{2}\frac{n}{\log_2 n}\\
&>\left(k^2+3k-(2k+3-t)t\right)(1- \frac{\epsilon}{k})\frac{n}{2}=(k-t)(k-t+3)(1- \frac{\epsilon}{k})\frac{n}{2}.
\end{split}
\end{equation*}
Using $\frac{\epsilon}{k}<\frac14$ and $k-t\geq 1$, this implies that $e(R_1,R_2)\geq \frac{3}{2}(k-t)n\geq (k-t+\frac{1}{2})n$.	
Applying Lemma \ref{LEM:Kss} on $(R_1,R_2)$ (with $s=k+1-t$ and $\delta=1/2$), we see that $(R_1,R_2)$ contains a copy of $K_{k+1-t,k+1-t}$ and thus a cycle of length $2k+2-2t$,
a contradiction. This proves the above upper bound of $e(V_1,V_2)$.
			
Combining the above inequalities, we have the following
$$mn + (1- \frac{\epsilon}{k})((k^2+3k)/2-m)n+\frac{n}4\geq e(V_1,V_2)= e(G) - e(G[V_1]) -e(G[V_2])$$
$$\geq \frac{1}{2}(k^2+3k+2)n - \left(5/8+ \frac{\epsilon}{k}(8k+2)\right)n-\frac{n}{8},$$
which implies that $\frac{1}{2}(k^2 +3k) - m \leq 8k+2$.
If $m\leq \frac{5k}{2}$, then we have $k^2 - 18k -4 \leq 0$, contradicting that $k \geq 19$.
This proves $m>\frac{5k}{2}$ and Claim \ref{Claim:I-m}.\footnote{If we assume that $e(G)\geq \frac{1}{2}(k^2+3k-2)n$ and $k\geq 150$ here instead,
then it will give $k^2-(13+4/\epsilon)k-(2m+4)\leq 0$. So still we can prove $m>5k/2$ and Claim \ref{Claim:I-m}.}
\end{proof}

\begin{claim}\label{Claim:I-cycles}
Let $t$ be a natural number. For any natural numbers $c_i$ for $i\in [t]$ satisfying that $\lceil c_1/2\rceil+\lceil c_2/2\rceil+...+\lceil c_t/2\rceil\leq m$,
there exist $t$ disjoint cycles $C_1,C_2,...C_t$ of lengths $2c_1,2c_2,...,2c_t$ in $G$ such that $|V(C_i) \cap V_2| \leq \lceil c_i/2 \rceil$.
\end{claim}

\begin{proof}[Proof of Claim \ref{Claim:I-cycles}]
We first show how to find a cycle $C_1$ of length $2c_1$ with $|V(C_1)\cap V_2|\leq \lceil c_1/2 \rceil$.
Take vertices $v_1,v_2,...v_{\lceil c_1/2 \rceil}$ in $V_2$ with $|N(v_i)\cap V_1|> (1-\frac{\epsilon}{k})n$.
Let $A_i = N(v_i) \cap N(v_{i+1}) \cap V_1$ for $1\leq i<\lceil c_1/2\rceil$ and $A_{\lceil c_1/2 \rceil} = N(v_1) \cap N(v_{\lceil c_1/2 \rceil}) \cap V_1$.
Let $B_i = V_1\setminus A_i$. So $|A_i|\geq (1-\frac{2\epsilon}{k})n$ and $|B_i|\leq \frac{2\epsilon}{k}n$.
We now assert that there are $\lceil c_1/2 \rceil$ disjoint paths $P_i$ of lengths two in $G[V_1]$ such that both endpoints of $P_i$ are in $A_i$ for each $i$.
Note that for each $i$, $G[B_i]$ doesn't contain $k$ cycles of consecutive even lengths, the shortest of which has length at most $2\log_{1+1/4k}|B_i| +2$.
So by Lemma \ref{LEM:main}, we have $e(B_i)\leq (8k+1)|B_i|/2\leq \frac{\epsilon}{k}(8k+1)n$. Then we have
$$(e(A_i)- |A_i|/2)+ (e(A_i,B_i) -|B_i|)= e(V_1) -e(B_i) -|A_i|/2-|B_i|$$
$$\geq \frac{\epsilon}{k}(8k+2)n+5n/8- \frac{\epsilon}{k}(8k+1)n-\left(\frac{\epsilon}{k}n+n/2\right)=n/8.$$
So for each $i$, either $e(A_i)\geq |A_i|/2+n/16$ or $e(A_i,B_i)\geq |B_i|+n/16$.
Since $m\leq k(k+3)/2$ and $G[V_1]$ has maximum degree at most $n/\log_2 n$, we have $n/16>3mn/ \log_2 n\geq \sum_{x\in V(P_1)\cup...\cup V(P_{i-1})} d_{G[V_1]}(x)$.
For all $i= 1,2, ..., \lceil c_1/2 \rceil$, in either case we can find a path $P_i$ of length two in $G[V_1]$ with both endpoints in $A_i$ and disjoint from $V(P_1)\cup...\cup V(P_{i-1})$.
No matter whether $c_1$ is even or odd, using these disjoint paths $P_i$ in $G[V_1]$ and vertices $v_1,v_2,...v_{\lceil c_1/2 \rceil}$,
it is easy to form a desired cycle $C_1$ of length $2c_1$ in $G$.

Because $G[V_1]$ has maximum degree at most $n/ \log_2 n$, repeatedly using the above arguments,
we can in fact find the desired cycle $C_i$ in $G[V_1,V_2)-(V(C_1)\cup...\cup V(C_{i-1}))$ for all $i\in [t]$.
This proves Claim \ref{Claim:I-cycles}.
\end{proof}	

From now on let $c_i=k+2-i$ for all $i\in [k]$.
Let $\ell$ be the maximum integer such that $\lceil c_1/2\rceil+\lceil c_2/2\rceil+...+\lceil c_l/2\rceil\leq m$.
By Claim \ref{Claim:I-m}, we have $m>\frac{5}{2}k$, which derives that $\ell\geq 5$.
	
\begin{claim}\label{Claim:I-e12}
$e(V_1,V_2)\leq \frac{1}{2}(k^2+3k)n - \frac{1}{4}(2k+1-\ell)\ell n$
\end{claim}

\begin{proof}[Proof of Claim \ref{Claim:I-e12}.]
Suppose for a contradiction that $e(V_1,V_2)> \frac{1}{2}(k^2+3k)n -\frac{\ell(2k+1-\ell)}{4}n $.
By Claim \ref{Claim:I-cycles}, $G$ contains $\ell$ disjoint cycles $C_1,C_2,...C_\ell$ with $|C_i|=2c_i$ and $|V(C_i)\cap V_2| \leq \lceil c_i/2\rceil$.
We may assume $\ell\leq k-1$.
Let $R_i=V_i-V(C_1) \cup V(C_2) ...\cup V(C_\ell)$ for $i\in \{1,2\}$. Then
\begin{equation*}
\begin{split}
e(R_1,R_2) &\geq e(V_1,V_2) - \sum_{i=1}^{\ell}\lceil c_i/2\rceil n-\sum_{i=1}^\ell 2c_i \cdot n/\log_2 n\\
&\geq e(V_1,V_2)-\sum_{i=1}^\ell \frac{k+3-i}{2} n= e(V_1,V_2) - \frac{\ell(2k+5-\ell)}{4}n\geq \frac{(k-\ell)(k-\ell+3)}{2}n
\end{split}
\end{equation*}
Let $s=\frac{(k-\ell)(k-\ell+3)}{2}$.
As $|R_1|\geq |R_2|^s$, by applying Lemma \ref{LEM:Kss} on $(R_1,R_2)$ with $\delta=1$,
we see that $(R_1,R_2)$ contains a copy of $K_{s,s}$ and thus $k-\ell$ disjoint cycles of lengths $4,6,...,2(k-\ell+1)$.
Together with the cycles $C_1,...,C_\ell$ as above, $G$ contains $k$ cycles of consecutive even lengths, a contradiction.
\end{proof}

Now we are ready to reach the final contradiction. By \eqref{equ:V2} and Claim \ref{Claim:I-e12}, we have
$$\frac{1}{2}(k^2 +3k +2)n \leq e(G)= e(V_1)+e(V_2)+e(V_1,V_2)$$
$$\leq \left(5/8+ \frac{\epsilon}{k}(8k+2)\right)n +n/8+\frac{1}{2}(k^2+3k)n - \frac14\ell(2k+1-\ell)n.$$
Using $\frac{\epsilon}{k}<1/4$, it implies that
$1+\ell(2k+1-\ell)\leq\frac{4\epsilon}{k}(8k+2)<8k+2.$
Since $k\geq \ell\geq 5$, we further have $1+5(2k-4)\leq 1+\ell(2k+1-\ell)<8k+2$.
This contradicts $k\geq 19$,\footnote{If we use that $e(G)\geq \frac{1}{2}(k^2+3k-2)n$ instead,
then the same analysis yields that $-7+5(2k-4)<8k+2$ and thus $k\leq 14$. So it also contradicts $k\geq 19$.}
completing the proof of Subsection \ref{subs:3.1}.
\qed

\subsection{$e(G[V_1])>(2k+1) n$}\label{subs:3.2}
Set $e(G[V_1]):=\left(\frac{2k+1}{2}+ \frac{\epsilon(8k+1)}{2}+\frac{1}{8}\right)n$.
By \eqref{equ:V1}, the average degree of $G[V_1]$ is at most $8k+1$,
so we have $\frac{1}{4}<\frac{8k+3}{4(8k+1)}<\epsilon<\frac{3k-\frac{1}{8}}{4k+\frac{1}{2}}<\frac{3}{4}$.

Suppose that $e(V_1,V_2)\leq \frac{1}{2}(k^2+3k)(1-\epsilon)n + 6\epsilon n$.
By \eqref{equ:V2}, we have
\begin{equation*}
\begin{split}
\frac{1}{2}(k^2 +3k +2)n&\leq e(G)=e(V_1)+e(V_2)+e(V_1,V_2)\\
&\leq \left(\frac{2k+1}{2}n + \frac{1}{2}\epsilon n(8k+1) +\frac{n}{8}\right) + \frac{n}{8} +\left((1-\epsilon)n\frac{1}{2}(k^2+3k) + 6\epsilon n\right)
\end{split}
\end{equation*}
Using $\epsilon>1/4$, we can get that
$2k-\frac{1}{2}\geq \epsilon(k^2-5k-13)>\frac{1}{4}(k^2-5k-13)$.
This implies that $k^2 -13k-11 \leq 0,$
a contradiction to $k\geq 19$.\footnote{If we use $e(G)\geq \frac{1}{2}(k^2+3k-2)n$ here instead,
then the same calculations give that $k^2-13k-27\leq 0$, which yields that $k\leq 16$.  So it also contradicts $k\geq 19$.}

Therefore, we have $e(V_1,V_2)>\frac{1}{2}(k^2+3k)(1-\epsilon)n + 6\epsilon n$.
Let $M$ denote the set of vertices $u\in V_2$ satisfying $|N(u)\cap V_1|>(1-\epsilon)n$. Let $m=|M|$.

We assert that $G[V_1,V_2)$ contains $m$ disjoint cycles of lengths $2k+2,2k,...,2k+4-2m$ such that any of them uses exactly one vertex in $V_2$ which is in $M$.
For any $u\in M$, let $A_u = N(u)\cap V_1$ and $B_u = V_1\setminus A_{u}$.
Suppose that $e[A_u,B_u)\leq \frac{2k+1}{2}n$.
By Lemma \ref{LEM:main}, we have $e(B_u)\leq \frac{1}{2}\epsilon n(8k+1)$.
This shows that
\begin{align}\label{equ:3.2}
e(G[V_1])\leq \frac{2k+1}{2}n + \frac{1}{2}\epsilon(8k+1)n= e(G[V_1]) - \frac{n}{8},
\end{align}
a contradiction.
So $e[A_u,B_u)> \frac{2k+1}{2}n$.
By the celebrated Erd\H{o}s-Gallai Theorem (see \cite{Erdos}),
$G[A_u,B_u)$ contains a cycle $D$ of length at least $2k+2$.
We first claim that $D$ contains a path of even length at least $2k$ with both endpoints in $A_u$.
To see this, if $D$ is odd, then clearly there exists an edge $xy\in E(D)$ with $x,y\in A_u$ and so $D-xy$ is such a path;
otherwise $D$ is even, then we have either $V(D)\subseteq A_u$ or $V(D)\cap B_u\neq \emptyset$ and in either case, we can find such a path easily.
Let $P= v_0v_1...v_\ell$ be such a path, where $\ell\geq 2k$ is even and $v_0,v_\ell\in A_u$.
If $v_1 \in B_u$, let $P'= v_2...v_\ell$; if $v_{\ell-1}\in B_u$, let $P'= v_0v_1...v_{\ell-2}$;
otherwise $v_1,v_{\ell-1}\in A_u$, let $P'= v_1...v_{\ell-1}$.
So $P'$ is a path of length $|P|-2$ with both endpoints in $A_u$.
Keeping this process, we can find a path $P_0\subseteq G[V_1]$ of length exactly $2k$ with both endpoints $x_0,y_0\in A_u$.
In this way, we can get a desired cycle $C_1:=P_0\cup x_0uy_0$ of length $2k+2$ in $G[V_1,V_2)$.
Now suppose we have obtained desired disjoint cycles $C_1,...,C_{i-1}$ for some $i\leq m$.
Since $\frac{n}{8}>|\cup_{j=1}^{i-1} V(C_j)|\cdot n/ \log_2n$,
by considering $G[V_1,V_2)-\cup_{j=1}^{i-1} V(C_j)$, we also can get a contradiction in the analog of \eqref{equ:3.2}
and then the same arguments enable us to find a desired cycle $C_i$ of length $2k+4-2i$.
This proves our assertion, that is, $G[V_1,V_2)$ contains $m$ disjoint cycles of lengths $2k+2,2k,...,2k+4-2m$
such that each of them uses exactly one vertex in $V_2$ which is in $M$.
Let $X$ be the union of vertex-sets of these $m$ cycles.

Let $C_1,C_2,...,C_t$ be a maximal collection of $t$ disjoint cycles in $G(V_1,V_2)-X$ with $|C_i|=2(k+2-m-i)$ for each $i\in [t]$.
Clearly we may assume that $t<k-m$, as otherwise, together with the above $m$ disjoint cycles,
there exist $k$ disjoint cycles of consecutive even lengths in $G$.
Also by our choice, $R:= G(V_1,V_2)-X\cup V(C_1)\cup ...\cup V(C_t)$ doesn't contain cycle of length $2(k-m-t+1)$.
Using \eqref{equ:V1} and $\frac{1}{4}<\epsilon<\frac{3}{4}$, it follows that
\begin{equation*}
\begin{split}	
e(R)&\geq e(V_1,V_2) - mn-\frac{1}{2}(2k+3-2m-t)t(1-\epsilon)n-(|X|+|C_1|+...+|C_t|)\cdot n/\log_2 n\\
&\geq \left(\frac{1}{2}(k^2+3k)(1-\epsilon)n + 6\epsilon n\right)-mn- \frac{1}{2}(2k+3-2m-t)t(1-\epsilon)n-\epsilon n\\
&=\frac{1}{2}((k-t)(k-t+3)+2mt)(1-\epsilon)n+5\epsilon n-mn\geq \left(\frac{1}{8}(k-t)(k-t+3)+\frac54-m\right)n.
\end{split}
\end{equation*}
Since $k-t\geq 1$, this implies that $e(R)\geq (k-m-t+\frac{1}{2})n$.
By Lemma \ref{LEM:Kss} (with $s=k-m-t+1$ and $\delta=1/2$), $R$ contains a copy of $K_{s,s}$ and thus a cycle of length $2s=2(k-m-t+1)$, a contradiction.
The proof of Theorem \ref{THM: main} now is completed.
\qedB

\section{Proof of Theorem \ref{THM:largek}}
The proof will be analogous to the one of Theorem \ref{THM: main}.
We shall only give detailed arguments for which different from Theorem \ref{THM: main} and sketch for the others.

Let $\epsilon>0$ be any real and $k\geq k_\epsilon$ be sufficiently large.
Let $G$ be a graph of sufficiently large order $n\geq n_k$ and average degree at least $k^2 +3k -2+\epsilon$.
Assume that $G$ doesn't contain $k$ disjoint cycles of consecutive even lengths.

Using the exactly same arguments as in the proof of Theorem \ref{THM: main},
we can get a partition $V(G)=V_1\cup V_2$ of $G$ satisfying \eqref{equ:V2} and \eqref{equ:V1}.

It is important to point out that in case that \eqref{equ:e1} holds (i.e., $e(G[V_1])>\frac{7}{8}n$),
we have put explanations in footnotes of Subsections \ref{subs:3.1} and \ref{subs:3.2} whenever the reasoning $e(G)\geq \frac12(k^2+3k+2)n$ was used therein,
by justifying that even under the weaker assumption $e(G)\geq \frac12(k^2+3k-2)n$,
the same arguments of Subsections \ref{subs:3.1} and \ref{subs:3.2} would also work (for $k\geq 150$ for instance).
So it would suffice to only consider that $e(G[V_1])\leq \frac{7}{8}n$ here.

Suppose that $e(G[V_1])<\frac{\epsilon}{4}n$.
Since $n$ is sufficiently large, the estimate of \eqref{equ:V2} can be improved to $e(G[V_2])\leq \frac{\epsilon}{8}n$.
Then we have
\begin{equation*}
e(V_1,V_2)\geq \frac{1}{2}(k^2+3k-2 +\epsilon)n-\frac{\epsilon}{4}n-\frac{\epsilon}{8}n=\left(\frac{k^2+3k}{2}-1+\frac{\epsilon}{8}\right)n.
\end{equation*}
By Lemma \ref{LEM:Kss}, $G(V_1,V_2)$ contains a $K_{s,s}$ with $s=\frac{1}{2}(k^2+3k)$.
Hence $G$ contains $k$ disjoint cycles of lengths $4,6,...,2k+2$, a contradiction.

It remains to consider that $\frac{\epsilon}{4}n \leq e(G[V_1])\leq \frac78n$.
Let $M$ be the set of all vertices $v\in V_2$ satisfying that $|N(v)\cap V_1| > (1-\frac{1}{k\sqrt{k}})n$ and let $m=|M|$.

Suppose $m\leq 2k$. Let $r:=\sum_{i=\lfloor \frac{k}{2}\rfloor}^{k+1}i\approx 3k^2/8$.
Using the bounds on $e(G[V_1])$ and $e(G[V_2])$, it is easy to get that $e(V_1,V_2)\geq r|V_1|$.
By Lemma \ref{LEM:Kss}, $G(V_1,V_2)$ contains a copy of $K_{r,r}$, which gives disjoint cycles of lengths $2k+2,2k,...,2 \lfloor \frac{k}{2} \rfloor$.
Let $R$ be obtained from $G(V_1,V_2)$ by deleting the vertices of these cycles.
For sufficient large $k$ and $n$, we have
\begin{equation*}
e(R)\geq e(G)-e(G[V_1])-e(G[V_2])-mn-(r-m)(1-1/k\sqrt{k})n - r|V_2| \geq \frac{1}{2}(k^2+3k-2r)n
\end{equation*}
By Lemma \ref{LEM:Kss}, $R$ contains a copy of $K_{s,s}$ with $s=\frac{1}{2}(k^2+3k)-r=2+3+...+(\lfloor \frac{k}{2}\rfloor-1)$.
Putting the above together, $G$ contains $k$ disjoint cycles of lengths $4,6,...,2k+2$, a contradiction.

So we have $m\geq 2k+1$.
We claim that there exists a cycle $C$ of length $2k$ or $2k+2$ in $G$ which uses at most $\ell:=\lfloor\frac{2k+2}{3}\rfloor +2$ vertices in $V_2$.
Fix $\ell$ vertices $v_1,v_2,...v_\ell$ in $V_2$ with $|N(v_i)\cap V_1|> (1-\frac{1}{k\sqrt{k}})n$.
Consider any $i,j\in [\ell]$. Let $A_{i,j} = N(v_i) \cap N(v_j) \cap V_1$ and $B_{i,j} = V_1\setminus A_{i,j}$.
So $|A_{i,j}|\geq(1-\frac{2}{k\sqrt{k}}) n$ and $|B_{i,j}|\leq \frac{2n}{k\sqrt{k}}$.
Since $G[B_{i,j}]$ does not contain $k$ cycles of consecutive even lengths, the shortest of which has length at most $2\log_{1+1/4k}n +2$.
By Lemma \ref{LEM:main}, $e(B_{i,j})\leq (4k+1/2)|B_{i,j}|\leq \frac{8k+1}{k\sqrt{k}}n$. Then for sufficiently large $k$, we have
\begin{equation*}
e(A_{i,j}) + e(A_{i,j},B_{i,j}) -|B_{i,j}|= e(G[V_1]) -e(B_{i,j}) -|B_{i,j}|\geq \frac{\epsilon}{4} n - \frac{8k+1}{k\sqrt{k}}n\geq \frac{\epsilon}{5} n.
\end{equation*}
So either $e(A_{i,j})\geq \epsilon n/10$ or $e(A_{i,j},B_{i,j})\geq |B_{i,j}|+\epsilon n/10$.
We call the pair $\{i,j\}$ {\it type I} in the former case and {\it type II} otherwise.
In view of \eqref{equ:V1}, after excluding any $2k$ vertices in $V_1$,
one can still find a path $P_{i,j}:=v_ixyv_j$ of length three with $x, y\in A_{i,j}$ if $\{i,j\}$ has type I,
and a path $Q_{i,j}:=v_ixyzv_j$ of length four with $x,z\in A_{i,j}$ and $y\in B_{i,j}$ if $\{i,j\}$ has type II.
For each $i\geq 1$, let $r_i$ be the maximum even integer not exceed the number of type I pairs $\{j,j+1\}$ for $1\leq j\leq i$,
and let $s_i$ be the number of type II pairs $\{j,j+1\}$ for $1\leq j\leq i$.
Denote $\alpha$ to be the minimum integer with $3r_\alpha+4s_\alpha\geq 2k-2$.
Since $3r_\alpha+4s_\alpha$ is even, one can infer that $3r_\alpha+4s_\alpha=2k-2$ or $2k$.
This provides a path $L$ between $v_1$ and some vertex say $v_\beta$ of length $2k-2$ or $2k$,
consisting of $r_\alpha$ many paths $P_{j,j+1}$ and $s_\alpha$ many paths $Q_{j,j+1}$ and using at most $\ell$ vertices in $V_2$.
Now it is easy to build the desired cycle $C$ by just adding a path between $v_1$ and $v_\beta$ of length two to the path $L$.

Let $R$ be the graph form $G(V_1,V_2)$ by deleting $V(C)$.
Then as $k$ is large, we have
\begin{align*}
e(R)\geq e(G)-e(G[V_1])-e(G[V_2])-\ell n - (2k+2)|V_2|\geq \frac{1}{2}(k^2+k)n
\end{align*}
By lemma~\ref{LEM:Kss}, $R$ contains a copy of $K_{s,s}$ with $s=\frac{1}{2}(k^2+k)$.
In either case that $|C|=2k$ or $|C|=2k+2$,
this together with the cycle $C$ can provide $k$ disjoint cycles of lengths $4,6,...,2k+2$.
We have finished the proof of Theorem \ref{THM:largek}.
\qedB

\medskip

We remark that in the current proof it is enough to choose $k_\epsilon=c/\epsilon^2$ for some large absolute constant $c$.

\section{Concluding remarks}
Our main result shows that Conjecture \ref{conj} holds for $k\geq 19$ and graphs of large order.
By some very careful calculations, this perhaps can be improved from $19$ to a smaller number.
However we believe our approach will not success for all $k\geq 2$.
For this reason, the case $k=2$ seems to be of particular interest,
where the conjecture suggests that average degree at least 12 would force the existence of two disjoint cycles of consecutive even lengths.
Being not able to prove it, we show the following weaker bound for $k=2$.

\begin{thm}\label{THM: k=2}
For every real $\epsilon >0$, there exists a number $n_\epsilon$ such that the following holds.
If $G$ is a graph of order at least $n_\epsilon$ and average degree at least $14+\epsilon$,
then $G$ contains two disjoint cycles of consecutive even lengths.
\end{thm}

{\noindent \it Proof.} Let $n_\epsilon$ be sufficiently large and $G$ be a graph with order $n\geq n_\epsilon$ and $e(G)\geq (7+\epsilon)n$.
Assume that $G$ does not contain two disjoint cycles of consecutive even lengths.
The proof is similar to the previous ones.
The same as the proof of Theorem \ref{THM: main}, we can find a partition $V(G)=V_1\cup V_2$
such that $G[V_1]$ does not contain two disjoint cycles of consecutive even lengths at most $2\log_{1+\epsilon/4} n+4$,
and moreover, $e(G[V_2])\leq |V_2|^2/2\leq n^{\frac{1}{10}}$.

By Lemma \ref{LEM:main}, we have $e(G[V_1])\le (5+\epsilon/2)n$.
So $e(V_1,V_2)\geq (2+\epsilon/4)n$.
Then by Lemma \ref{LEM:Kss}, $G(V_1,V_2)$ contains a copy of $K_{3,3}$.
This shows that there exist copies of $C_4$ and $C_6$ in $G$.
Recall the definition of the sequence $G:=G_0\supseteq G_1 \supseteq ...\supseteq  G_m$,
where at each time, two consecutive even cycles with minimum lengths in $G_i$ would be put in $V_2$.
This means that $G[V_2]$ must contain a copy of $C_6$. So $G[V_1]$ cannot contain any copy of $C_4$.

Let $u\in V_2$ be a vertex with the maximum number of neighbors in $V_1$.
Let $A= N(u)\cap V_1$ and $B=V_1\setminus A$ with $|A|=(1-\alpha)n$ .

We claim that $e(G[V_1])\leq (5+\epsilon/2)\alpha n+n+\frac{3}{2}(1-\alpha)n$.
Suppose for a contradiction that $e(G[V_1])>(5+\epsilon/2)\alpha n+n+\frac{3}{2}(1-\alpha)n$.
By Lemma \ref{LEM:main}, $e(G[B])\le(5+\epsilon/2)\alpha n$.
Then either $e(G[A])> \frac{3}{2}|A|$ or $e(A,B)> n$.
If $e(A)> \frac{3}{2}|A|$, by Erd\H{o}s-Gallai Theorem \cite{Erdos}, $G[A]$ contains a path on five vertices;
otherwise $e(A,B)> n$, $G(A,B)$ contains an even cycle of length at least six (as $G[V_1]$ has no four-cycle).
So in either case, $G[V_1\cup\{u\}]$ contains a cycle $C$ of length six.
Let $R$ be obtained from $G(V_1,V_2)$ by deleting $V(C)$.
Then $e(R)\geq e(V_1,V_2)-n-5|V_2|\geq (2+\epsilon/4)n-n-5|V_2|\geq (1+\epsilon/5)n$.
By Lemma \ref{LEM:Kss}, $R$ contains a copy of $C_4$,
which together with the cycle $C$ give a contradiction.
This proves the claim.

Recall that $G(V_1,V_2)$ contains a cycle $D$ of length six.
Let $R'$ be obtained from $G(V_1,V_2)$ by deleting $V(D)$.
By the claim, we have
\begin{equation*}
\begin{split}
e(R')&\geq e(G)-e(G[V_1])-e(G[V_2])- 3(1-\alpha)n- 3|V_2|\\
&\geq \left(\frac{3+2\epsilon}{2}-\frac{1+\epsilon}{2}\alpha\right) n-n^{1/2}\geq (1+\epsilon/4)n.
\end{split}
\end{equation*}
By Lemma \ref{LEM:Kss}, $R'$ contains a copy of $C_4$, again a contradiction.
This finishes the proof.
\qedB

\medskip

To conclude, we would like to ask whether for any $k\geq 2$ and any real $\epsilon>0$,
there exists $n_{k,\epsilon}$ such that every graph of order at least $n_{k,\epsilon}$ and average degree at least $k^2+3k-2+\epsilon$ contains $k$ disjoint cycles of consecutive even lengths.

\end{document}